\documentclass[a4paper,10pt]{amsart}
\usepackage{amsmath,amssymb}
\usepackage{amscd}
\newtheorem{thm}{Theorem}[section]
\newtheorem{prop}[thm]{Proposition}
\newtheorem{lemma}[thm]{Lemma}
\newtheorem{cor}[thm]{Corollary}

\theoremstyle{remark}

\newcommand{\id}{{\rm{id}}}
\newcommand{\Ad}{{\rm{Ad}}}

\newcommand{\BN}{\mathbf N}
\newcommand{\BR}{\mathbf R}
\newcommand{\BC}{\mathbf C}

\newcommand{\BK}{\mathbf K}
\newcommand{\BB}{\mathbf B}
\newcommand{\la}{\langle}
\newcommand{\ra}{\rangle}

\newcommand{\Aut}{{\rm{Aut}}}

\newtheorem{Def}{Definition}[section]

\title{Strong Morita equivalence for inclusions of $C^*$-algebras induced by
twisted actions of a countable discrete group}
\author{Kazunori Kodaka}
\address{Department of Mathematical Sciences, Faculty of Science, Ryukyu
University,
\endgraf
Nishihara-cho, Okinawa, 903-0213, Japan}

\address{\sl{E-mail address}: \rm{kodaka@math.u-ryukyu.ac.jp}}

\begin{document}

\begin{abstract}
We consider two twisted actions of a countable discrete group on $\sigma$-unital $C^*$-algebras.
Then by taking the reduced crossed products, we get two inclusions of $C^*$-algebras.
We suppose that they
are strongly Morita equivalent as inclusions of
$C^*$-algebras. Also, we suppose that one of the inclusions of $C^*$-algebras is irreducible, 
that is, the relative commutant of one of the $\sigma$-unital $C^*$-algebra in the multiplier
$C^*$-algebra of the reduced twisted crossed product is trivial. We show that
the two actions are then strongly Morita equivalent up to some automorphism of the
group.
\end{abstract}

\maketitle

\section{Introduction}\label{sec:intro} In the previous papers \cite{KT3:equivalence}, \cite{KT4:morita},
we discussed strong Morita equivalence for twisted coactions of a finite dimensional $C^*$-Hopf algebra
on unital $C^*$-algebras and unital inclusions of unital $C^*$-algebras. Also, in \cite {KT5:inclusion2}, we
clarified the relation between strong Morita equivalence for such twisted coactions and
strong Morita equivalence for the associated unital inclusions of unital $C^*$-algebras.
In the present paper, we shall discuss the same subject as above in the case of twisted
actions of a countable discrete group on $\sigma$-unital $C^*$-algebras.
\par
Let $(\alpha, w_{\alpha})$ and $(\beta, w_{\beta})$ be twisted actions of a countable discrete
group $G$ on $\sigma$-unital $C^*$-algebras $A$ and $B$, and let $A\times_{\alpha, w_{\alpha}, r}G$ and
$B\rtimes_{\beta, w_{\beta}, r}G$ denote the twisted reduced
crossed products of $A$ and $B$ by $(\alpha, w_{\alpha})$ and $(\beta, w_{\beta})$, respectively.
Then we obtain the inclusions of $C^*$-algebras
$A\subset A\rtimes_{\alpha, w_{\alpha}, r}G$ and $B\subset B\rtimes_{\beta, w_{\beta}, r}G$.
In this paper, we first review that
if $(\alpha, w_{\alpha})$ and $(\beta, w_{\beta})$ are
strongly Morita equivalent, then in the same way as in 
Combes \cite {Combes:morita} or Curto, Muphy and Williams
\cite {CMW:equivalence}, the inclusions $A\subset A\rtimes_{\alpha, w_{\alpha}, r}G$
and $B\subset B\rtimes_{\beta, w_{\beta}, r}G$
are strongly Morita equivalent. However, it is not clear that the converse result holds in general.
Our main result, Theorem \ref{thm:strong4}, is as follows:
We suppose that the inclusion $A\subset A\rtimes_{\alpha, w_{\alpha}, r}G$ is irreducible, that is,
$A' \cap M(A\rtimes_{\alpha, w_{\alpha}, r}G)=\BC 1$.
If the inclusions $A\subset A\rtimes_{\alpha,w_{\alpha}, r}G$
and $B\subset B\rtimes_{\beta, w_{\beta}, r}G$ are strongly Morita equivalent,
then there is an automorphism $\phi$ of $G$
such that $(\alpha^{\phi}, w_{\alpha}^{\phi})$ and $(\beta, w_{\beta})$ are
strongly Morita equivalent, where $(\alpha^{\phi}, w_{\alpha}^{\phi})$ is
the twisted action of $G$ on $A$
defined by
$\alpha_t^{\phi}(a)=\alpha_{\phi(t)}(a)$ and $w_{\alpha}^{\phi}(t, s)=w_{\alpha}(\phi(t), \phi(s))$
for any $t, s\in G$, $a\in A$. This result is similar to \cite [Theorem 6.2]{KT5:inclusion2}.
We remark that the condition $A' \cap M(A\rtimes_{\alpha,w_{\alpha}, r}G)=\BC1$ holds
if and only if
$(\alpha , w_{\alpha})$ is a free twisted action of $G$ on $A$ (cf. Corollary \ref{cor:free5}).
This freeness of $(\alpha, w_{\alpha})$ plays an important role in Section \ref{sec:inclusion}.
\par
We let $\BK$ denote the $C^*$-algebra of all compact operators on a countably infinite dimensional Hilbert space
and $\{e_{ij}\}_{ij\in\BN}$ its system of matrix units. By a homomorphism between two $C^*$-algebras, we will
always mean a *-homomorphism. This also applies to isomorphisms and automorphisms.
For each 
$C^*$-algebra $A$, we denote by $M(A)$ the multiplier $C^*$-algebra of $A$.
Let $\pi$ be an isomorphism of $A$ onto
a $C^*$-algebra $B$. Then there is a unique strictly continuous isomorphism of $M(A)$ onto $M(B)$
extending $\pi$ by Jensen and Thomsen \cite [Corollary 1.1.15]{JT:KK}. We denote it by $\underline{\pi}$.
For an algebra $A$, we denote by $\id_A$ the identity map on $A$; if $A$ is unital, we denote by $1_A$ the
unit of $A$.
If no confusion arises, we denote them by $\id$ and $1$, respectively.
Throughout this paper, we denote by $G$ a countable discrete group with unit element $e$.
 
\section{Preliminaries}\label{sec:pre} First, we give some definitions.
Let $A$ be a $C^*$-algebra
and $G$ a countable discrete group with the unit element $e$.
Let $\Aut(A)$ be the group of all automorphisms of $A$ and $U(M(A))$ the group of all unitary elements
in $M(A)$.

\begin{Def}\label{def:pre1} By a 
\sl
twisted action
\rm
$(\alpha, w_{\alpha})$ of $G$ on $A$, we mean a map $\alpha$ from $G$ to $\Aut(A)$ and a
map $w_{\alpha}$ from $G\times G$ to $U(M(A))$ satisfying following:
\newline
(1) $\alpha_t \circ \alpha_s =\Ad (w_{\alpha}(t, s))\circ \alpha_{ts}$,
\newline
(2) $w_{\alpha}(t, s)w_{\alpha}(ts, r)=\underline{\alpha_t} (w_{\alpha}(s, r))w_{\alpha}(t, sr)$,
\newline
(3) $w_{\alpha}(t, e)=w_{\alpha}(e, t)=1_{M(A)}$
\newline
for any $t, s, r\in G$.
\end{Def}

By easy computations, for any $t\in G$,
\begin{align*}
& \alpha_e =\id, \quad w_{\alpha}(t, t^{-1} )=\underline{\alpha_t} (w_{\alpha}(t^{-1}, t)) , \\
& \alpha_t^{-1}=\alpha_{t^{-1}}\circ\Ad(w_{\alpha}(t, t^{-1})^* )=\Ad(w_{\alpha}(t^{-1}, t)^* )\circ\alpha_{t^{-1}}
\end{align*}

Let $(\alpha, w_{\alpha})$ be a twisted action of $G$ on a $C^*$-algebra $A$. Then
we get a twisted action $(\underline{\alpha}, w_{\alpha})$ of $G$ on $M(A)$ such that
$\underline{\alpha_t}$ is the unique strictly continuous automorphism of $M(A)$ extending
$\alpha_t $ to $M(A)$ for any $t\in G$.
\par
We regard $A$ and $M(A)$ as $C^*$-subalgebras of $A\rtimes_{\alpha, w_{\alpha}, r}G$ and
$M(A)\rtimes_{\underline{\alpha}, w_{\alpha}, r}G$ in the usual way, respectively.
By considering the universal representation of $A$ on a Hilbert space $\mathcal{H}_u$, we may regard
$M(A)\rtimes_{\underline{\alpha}, w_{\alpha}, r}G$ and $A\rtimes_{\alpha, w_{\alpha}, r}G$ as 
acting non-degenerately on the same Hilbert space $l^2 (G, \mathcal{H}_u )$.
Then we can see that $M(A)\rtimes_{\underline{\alpha}, w_{\alpha}, r}G
\subset M(A\rtimes_{\alpha, w_{\alpha}, r}G)$ (cf. Pedersen \cite[Sections 3.12 and 7.7]{Pedersen:auto}
for the case of an untwisted action).
Thus we have the following inclusions of $C^*$-algebras:
\begin{align*}
A & \subset A\rtimes_{\alpha, w_{\alpha}, r}G \subset M(A\rtimes_{\alpha, w_{\alpha}, r}G) \\
M(A) & \subset M(A)\rtimes_{\underline{\alpha}, w_{\alpha}, r}G
\subset M(A\rtimes_{\alpha, w_{\alpha}, r}G) .
\end{align*}
We may also form $(\alpha\otimes\id_{\BK} \, , \, w_{\alpha}\otimes 1_{M(\BK)})$,
the twisted action of $G$ on $A\otimes\BK$. As is well-known, 
there is an isomorphism of $(A\otimes\BK)\rtimes_{\alpha\otimes\id, w_{\alpha}\otimes 1, r}G$ onto
$(A\rtimes_{\alpha, w_{\alpha},r}G)\otimes\BK$ such that its restriction to $A\otimes\BK$ is the identity
map on $A\otimes\BK$. Thus
$$
A\otimes\BK\subset(A\otimes\BK)\rtimes_{\alpha\otimes\id, w_{\alpha}\otimes 1, r}G
$$ 
and
$$
A\otimes\BK\subset(A\rtimes_{\alpha, w_{\alpha}, r}G)\otimes\BK
$$
are isomorphic as inclusions of $C^*$-algebras. We identify them as inclusions of $C^*$-algebras
in this paper.
\par
It is also well-known that there exists a canonical faithful conditional expectation from
$A\rtimes_{\alpha, w_{\alpha}, r}G$ onto $A$. The definition of a conditional expectation can be found
in Blackadar \cite [II. 6.10]{Blackadar:operator}. One way to construct $E^A$ is as follows.
\par
Let $E^{M(A)}$ be the faithful canonical conditional expectation from 
$M(A)\rtimes_{\underline{\alpha}, w_{\alpha}, r}G$ onto $M(A)$ defined in B\'edos and Conti
\cite [Section 3]{BC:discrete}. Then we may let $E^A$ be the restriction
of $E^{M(A)}$ to $A\rtimes_{\alpha, w_{\alpha}, r}G$,
that is, $E^A =E^{M(A)}|_{A\rtimes_{\alpha, w_{\alpha}, r}G}$.
\par
For each $t\in G$, let $\delta_t $ be the function from $G$ to $M(A)$ defined by
$$
\delta_t (s) = \begin{cases} 1_{M(A)} & \text{if $s=t$} \\
0 & \text{if $s\ne t$} .
\end{cases}
$$
Since $\delta_t \in l^1 (G, M(A))$, $\delta_t\in M(A)\rtimes_{\underline{\alpha}, w_{\alpha}, r}G$.
We regard $\delta_t$ as an element in $M(A\rtimes_{\alpha, w_{\alpha},  r}G)$
by the above inclusion $M(A)\rtimes_{\underline{\alpha}, w_{\alpha}, r}G \subset M(A\rtimes_{\alpha, w_{\alpha}, r}G)$.
\par
Let $\{u_i \}_{i\in I}$ be an approximate unit of $A$ with $||u_i||\leq 1$ for any $i\in I$.
We fix the approximate unit $\{u_i \}_{i\in I}$ of $A$ in this paper.
Let $x\in M(A\rtimes_{\alpha, w_{\alpha}, r}G)$. Then $\{E^A (x(u_i \delta_t^* ))\}_{i\in I}$
is a Cauchy net in $M(A)$ under the strict topology
for any $t\in G$ since $\alpha_t (a)=\delta_t a\delta_t^*$ for any $t\in G$ and $a\in A$.
Indeed, for any $i, j\in I$, $a\in A$,
\begin{align*}
||a(E^A (xu_i \delta_t^* )-E^A  (xu_j \delta_t^* ))|| & =||aE^A (x\delta_t^* (\alpha_t (u_i )-\alpha_t (u_j )))|| \\
& =||E^A (ax\delta_t^* (\alpha_t (u_i )-\alpha_t (u_j )))|| \\
& =||E^A (ax\delta_t^* )(\alpha_t (u_i )-\alpha_t (u_j ))|| \\
& =||\alpha_t ^{-1}(E^A (ax\delta_t^* ))(u_i -u_j)||\to 0 \quad(i, j\to\infty) .
\end{align*}
Similarly, $||(E^A (xu_i \delta_t^* )-E^A (xu_j \delta_t^* ))a||\to 0$ $(i, j\to\infty)$.
Let $x_t =\lim_i E^A (xu_i \delta_t^* )$ for every $x\in M(A\rtimes_{\alpha, w_{\alpha}, r}G)$ and
$t\in G$, where the limit is taken under the 
strict topology in $M(A)$. Then $x_t \in M(A)$ for all $t\in G$.

\begin{Def}\label{def:pre5} Let $x\in M(A\rtimes_{\alpha, w_{\alpha}, r}G)$.
For any $t\in G$, let $x_t$ be the element in $M(A)$ defined
in the above. We call $x_t $ the {\sl Fourier coefficient} of
$x\in M(A\rtimes_{\alpha, w_{\alpha}, r}G)$ at $t\in G$,
and call $\{x_t\}_{t\in G}$ the {\sl Fourier coefficients} of $x\in M(A\rtimes_{\alpha, w_{\alpha}, r}G)$.
\end{Def}

We note that if $x\in A\rtimes_{\alpha, w_{\alpha}, r}G$, then $x_t =E^A (x\delta^* )$ is the usual
Fourier coefficient of $x$ at $t\in G$ and  that if $x\in M(A)$, then
$$
x_t = \begin{cases} x & \text{if $t=e$} \\
0 & \text{if $t\ne e$}\end{cases} .
$$

\begin{lemma}\label{lem:pre6} Let $x\in A\rtimes_{\alpha, w_{\alpha}, r}G$. 
If $x_t =0$
for all $t\in G$, then $x=0$.
\end{lemma}
\begin{proof} For any $a\in A$, $t\in G$,
$$
0=E^A (x\delta_t^* )\alpha_t(a)=E^A (x\delta_t^* \alpha_t (a))=E^A (xa\delta_t^* )
$$
since $\alpha_t (a)=\delta_t a\delta_t^*$.
Since $A\rtimes_{\alpha, w_{\alpha},  r}G$ is the closed linear span of
$\{a\delta_t^* \, | \, a\in A\, , \, t\in G \}$, $E^A (xy)=0$
for all $y\in A\rtimes_{\alpha, w_{\alpha}, r}G$. Since $E^A $ is faithful, $x=0$.
\end{proof}

\begin{lemma}\label{lem:pre7} Let $x\in M(A\rtimes_{\alpha, w_{\alpha}, r}G)$. If
$x_t =0$ for all $t\in G$, then $x=0$.
\end{lemma}
\begin{proof} For any $a\in A$, $t\in G$,
$$
0=x_t \alpha_t (a)=\lim_i E^A (x(u_i \delta_t^* ))\alpha_t (a)=\lim_i E^A (x(u_i a )\delta_t^* )
=E^A (xa\delta_t^* ) .
$$
since $\alpha_t (a)=\delta_t a\delta_t^*$. Since $E^A (xa\delta_t^* )$ is the Fourier coefficient of
the element $xa$ at
$t\in G$, by Lemma \ref{lem:pre6}, $xa=0$ for any $a\in A$. Thus for any $a\in A$, $t\in G$,
$xa\delta_t =0$. Since $A\rtimes_{\alpha, w_{\alpha},  r}G$ is the closed linear span of
$\{a\delta_t \, | \, a\in A\, , \, t\in G \}$, $xy=0$ for any $y\in A\rtimes_{\alpha, w_{\alpha}, r}G$.
We note that by \cite [Proposition 3.12.3]{Pedersen:auto},
$x$ can be regarded as a double centralizer on $A\rtimes_{\alpha, w_{\alpha}, r}G$
by multiplication. Hence $x=0$.
\end{proof}

\begin{Def}\label{def:pre2} Let $(\alpha, w_{\alpha})$ and $(\beta, w_{\beta})$ be twisted actions of $G$ on
a $C^*$-algebra $A$. We say that $(\alpha, w_{\alpha})$ and $(\beta, w_{\beta})$ are
\sl
exterior equivalent
\rm
if there are unitary elements $\{w_t \}_{t\in G}\subset M(A)$ satisfying the following:
\newline
(1) $\beta_t =\Ad(w_t )\circ\alpha_t$,
\newline
(2) $w_{ts}=w_{\beta}(t, s)^* w_t \underline{\alpha_t }(w_s )w_{\alpha}(t, s)$
\newline
for any $t, s\in G$.
\end{Def}

We refer to Packer and Raeburn
\cite {PR:twisted} for more details about this notion.

Let $A$ and $B$ be $C^*$algebras. Let $X$ be an $A-B$-equivalence bimodule
(see Gracia-Bond\'{i}a, V\'{a}rilly and Figueroa \cite {GVF:Noncommutative}). Note that
$X$ is often called an $A-B$-imprimitivity bimodule (see e.g. Raeburn and Williams \cite {RW:continuous}).
For any $a\in A$, $b\in B$,
$x\in X$, we denote by $a\cdot x$ the left $A$-action on $X$ and by $x\cdot b$ the right $B$-action on
$X$. Let $\widetilde{X}$ be the dual $B-A$-equivalence bimodule of $X$ and let $\widetilde{x}$ denote
the element in $\widetilde{X}$ associated to an element $x\in X$. Also, we regard $X$ as a Hilbert
$M(A)-M(B)$-bimodule in the sense of Brown, Mingo and Shen \cite {BMS:quasi} as follows:
Let $\BB_B (X)$ be the $C^*$-algebra of all adjointable right $B$-linear operators on $X$.
We recall that an adjointable right $B$-linear operator on $X$ is bounded. Then
$\BB_B (X)$ can be identified with $M(A)$. Similarly let ${}_A \BB(X)$ be the $C^*$-algebra
of all adjointable left $A$-linear operators on $X$. Then ${}_A \BB(X)$ can be
identified with $M(B)$. Using these identifications, we may regard $X$ as a Hilbert $M(A)-M(B)$-bimodule.
Let $\Aut(X)$ be the group of all bijective bounded linear maps on $X$.

\begin{Def}\label{def:pre3} Let $(\alpha, w_{\alpha})$ and $(\beta, w_{\beta})$ be twisted actions of $G$ on
$C^*$-algebra $A$ and $B$, respectively. We say that $(\alpha, w_{\alpha})$ and $(\beta, w_{\beta})$ are
\sl
strongly Morita equivalent
\rm
if there are an $A-B$-equivalence bimodule $X$ and a map $\lambda$ from $G$ to $\Aut (X)$ satisfying
the following:
\newline
(1) $\alpha_t ({}_A \la x, y \ra)={}_A \la \lambda_t (x) \, , \, \lambda_t (y) \ra$,
\newline
(2) $\beta_t (\la x, y \ra_B )=\la \lambda_t (x) \, , \, \lambda_t (y) \ra_B$,
\newline
(3) $(\lambda_t \circ \lambda_s )(x)=w_{\alpha}(t, s)\cdot \lambda_{ts}(x)\cdot w_{\beta}(t, s)^* $
\newline
for any $t, s\in G$, $x, y\in X$, where we regard $X$ as a Hilbert $M(A)-M(B)$-bimodule in the above way.
\end{Def}

We note that if the twisted actions $(\alpha, w_{\alpha})$ and $(\beta, w_{\beta})$ of $G$ on $A$
are exterior equivalent, then
they are strongly Morita equivalent. Indeed,
let $\{w_t \}_{t\in G}$ be unitary elements in $M(A)$ satisfying Conditions (1), (2) in Definition \ref{def:pre2}.
We regard $A$ as the trivial $A-A$-equivalence bimodule in the usual way and we denote it by $X_0$.
For any $t\in G$, let $\lambda_t \in \Aut (X_0 )$ be
defined by
$$
\lambda_t (x)=\alpha_t (x)w_t^*
$$
for all $x\in X_0$. Then one can easily check that $\lambda$ satisfies Conditions (1)-(3) in Definition \ref{def:pre3}.
\par
We recall the definition of strong Morita equivalence for inclusions of $C^*$-algebras.

\begin{Def}\label{def:pre3-2} Two inclusions of $C^*$-algebras $A\subset C$ and $B\subset D$ with
$\overline{AC}=C$
and $\overline{BD}=D$ are said to be
\sl
strongly Morita equivalent
\rm
if there exists a $C-D$-equivalence bimodule $Y$ having a closed subspace $X$ satisfying these as six
conditions:
\newline
(1) $a\cdot x\in X$, ${}_C \la x, y \ra \in A$ for any $a\in A$, $x, y\in X$ and $\overline{{}_C \la X, X \ra}=A$,
$\overline{{}_C \la Y, X \ra}=C$,
\newline
(2) $x\cdot b\in X$, $\la x, y \ra_B \in B$ for any $b\in B$, $x, y\in X$ and $\overline{\la X, X \ra_D} =B$,
$\overline{\la Y, X \ra_D }=D$.
\end{Def}

For each twisted action $(\alpha, w_{\alpha})$ of $G$ on a $C^*$-algebra $A$, we get an inclusion
of $C^*$-algebras $A\subset A\rtimes_{\alpha, w_{\alpha}, r}G$ with
$\overline{A(A\rtimes_{\alpha, w_{\alpha}, r}G)}=A\rtimes_{\alpha, w_{\alpha}, r}G$.
\par
Let  $(\alpha, w_{\alpha})$ and $(\beta, w_{\beta})$ be twisted actions of $G$ on $A$ and $B$, respectively.
We suppose that they are strongly Morita equivalent so that there are an $A-B$-equivalence
bimodule $X$ and a map $\lambda$ from $G$ to $\Aut (X)$ satisfying Conditions (1)-(3)
in Definition \ref{def:pre3}.
We show that the inclusions of $C^*$-algebras $A\subset A\rtimes_{\alpha, w_{\alpha}, r}G$ and
$B\subset B\rtimes_{\beta, w_{\beta}, r}G$ are strongly Morita equivalent in the same way as in Combes
\cite {Combes:morita} or Curto, Muhly and Williams \cite {CMW:equivalence}. Let $L_X$ be the linking
$C^*$-algebra for $X$ defined by
$$
L_X =\begin{bmatrix} A & X \\
\widetilde{X} & B \end{bmatrix}
=\{ \begin{bmatrix} a & x \\
\widetilde{y} & b \end{bmatrix} \, | \, a\in A \, , b\in B, \, x, y\in X \} .
$$
For more details about linking $C^*$-algebras, we refer
to \cite [Section 3.2]{RW:continuous}.
Let $t, s\in G$. Let $\gamma_t $ be the map on $L_X$ defined by
$$
\gamma_t (\begin{bmatrix} a & x \\
\widetilde{y} & b \end{bmatrix})=\begin{bmatrix} \alpha_t (a) & \lambda_t (x) \\
\widetilde{\lambda_t (y)} & \beta_t (b) \end{bmatrix}
$$
for all $\begin{bmatrix} a & x \\
\widetilde{y} & b \end{bmatrix}\in L_X$.
Also, let $w_{\gamma}(t, s)$ be the unitary element in $M(L_X )$ defined by
$$
w_{\gamma}(t, s)=\begin{bmatrix} w_{\alpha} (t, s) & 0 \\
0 & w_{\beta}(t, s) \end{bmatrix}
$$
for all $t, s\in G$.
Then by routine computations, one verifies that $(\gamma, w_{\gamma})$ is a twisted
action of $G$ on $L_X$. Let $L_X \rtimes_{\gamma, w_{\gamma}, r}G$ be the twisted reduced
crossed product of $L_X$ by $(\gamma, w_{\gamma})$. Let $p=\begin{bmatrix} 1 & 0 \\
0 & 0 \end{bmatrix}$ and $q=\begin{bmatrix} 0 & 0 \\
0 & 1 \end{bmatrix}$. Then $p$ and $q$ are projections in
$M(L_X )\subset M(L_X \rtimes_{\gamma, w_{\gamma}, r}G)$ and they are full in
$M(L_X \rtimes_{\gamma, w_{\gamma}, r}G)$. Let $Y=p(L_X \rtimes_{\gamma, w_{\gamma}, r}G)q$
and $X=pL_X q$. Then clearly $X$ is a closed subspace of $Y$ and $Y$ is an
$A\rtimes_{\alpha, w_{\alpha}, r}G-B\rtimes_{\beta, w_{\beta}, r}G$-equivalence bimodule by easy computations,
where $A\rtimes_{\alpha, w_{\alpha}, r}G$ and $B\rtimes_{\beta, w_{\beta}, r}G$ are identified with
$p(L_X \rtimes_{\gamma, w_{\gamma}, r}G)p$ and $q(L_X \rtimes_{\gamma, w_{\gamma}, r}G)q$,
respectively. Furthermore, identifying $A$ and $B$ with $pL_Xp$ and $qL_X q$, respectively,
one readily obtains that the inclusions $A\subset A\rtimes_{\alpha, w_{\alpha}, r}G$ and
$B\subset B\rtimes_{\beta, w_{\beta}, r}G$ are strongly Morita equivalent with respect to
$Y$ and $X$.
Hence we obtain the following proposition.

\begin{prop}\label{prop:pre4} Let $(\alpha, w_{\alpha})$ and $(\beta, w_{\beta})$ be twisted actions of
a countable discrete group $G$
on $C^*$-algebras $A$ and $B$, respectively. If $(\alpha, w_{\alpha})$ and $(\beta, w_{\beta})$ are
strongly Morita equivalent up to some automorphism of $G$, then the inclusions of $C^*$algebras
$A\subset A\rtimes_{\alpha, w_{\alpha}, r}G$
and $B\subset B\rtimes_{\beta, w_{\beta}, r}G$ are strongly Morita equivalent.
\end{prop}
\begin{proof} Let $\phi$ be an automorphism of $G$. Let $(\alpha^{\phi}, w_{\alpha}^{\phi})$ be 
the twisted action of $G$ on $A$ defined by
$$
\alpha_t^{\phi}(a)=\alpha_{\phi(t)}(a) \, \quad w_{\alpha}^{\phi}(t, s)=w_{\alpha}(\phi(t), \phi(s))
$$
for any $a\in A$, $t, s\in G$. Then there is an isomorphism $\pi$ of $A\rtimes_{\alpha, w_{\alpha}, r}G$ onto
$A\rtimes_{\alpha^{\phi}, w_{\alpha}^{\phi}, r}G$ with $\pi|_{A}=\id$ on $A$.
Thus we may assume that $(\alpha, w_{\alpha})$ and $(\beta, w_{\beta})$ are strongly
Morita equivalent. Hence by the above discussion, we set that $A\subset A\rtimes_{\alpha, w_{\alpha}, r}G$
and $B\subset B\rtimes_{\beta, w_{\beta}, r}G$ are strongly Morita equivalent.
\end{proof}

\section{Relative commutants of inclusions of $\sigma$-unital $C^*$-algebras}\label{sec:com}
Let $A$ be a $C^*$-algebra and $p$ a projection in $M(A\otimes\BK)$.
We note that by \cite [E 1.1.8]{JT:KK}, there is an isomorphism of
$M(p(A\otimes\BK)p)$ onto $pM(A\otimes\BK)p$ such that its restriction to
$p(A\otimes\BK)p$ is the identity map on $p(A\otimes\BK)p$. We identify
$M(p(A\otimes\BK)p)$ with $pM(A\otimes\BK)p$ by the above isomorphism.
We begin this section with the following lemma:

\begin{lemma}\label{lem:com1} Let $A\subset C$ be an inclusion of $C^*$-algebras with $\overline{AC}=C$.
Then the following conditions are equivalent:
\newline
$(1)$ $A' \cap M(C)=\BC1$,
\newline
$(2)$ $(A\otimes\BK)' \cap M(C\otimes\BK)=\BC1$.
\end{lemma}
\begin{proof} $(1)\Rightarrow (2)$: Let $x\in (A\otimes\BK)' \cap M(C\otimes\BK)$. Then
for any $a\in A$, $k\in\BK$,
$$
x(a\otimes k)=(a\otimes k)x .
$$
Since $A$ is dense in $M(A)$ under the strict topology, $x(1\otimes k)=(1\otimes k)x$
for any $k\in \BK$.
We note that
\begin{align*}
(1\otimes e_{11})M(C\otimes\BK)(1\otimes e_{11})
& =M((1\otimes {e_{11}})(C\otimes\BK)(1\otimes e_{11})) \\
& =M(C\otimes e_{11})=M(C)\otimes e_{11}.
\end{align*}
For any $a\in A$,
$$
(1\otimes e_{11})x(1\otimes e_{11})(a\otimes e_{11})=(a\otimes e_{11})(1\otimes e_{11})x(1\otimes e_{11}) .
$$
Since $A' \cap M(C)=\BC1$, there is a $c\in \BC$ such that
$(1\otimes e_{11})x(1\otimes e_{11})=c(1\otimes e_{11})$.
Thus $x(1\otimes e_{11})=c(1\otimes e_{11})$. For any $n\in\BN$,
\begin{align*}
x(1\otimes e_{nn}) & =x(1\otimes e_{n1})(1\otimes e_{11})(1\otimes e_{1n}) \\
& =(1\otimes e_{n1})x(1\otimes e_{11})(1\otimes e_{1n}) \\
& =(1\otimes e_{n1})c(1\otimes e_{11})(1\otimes e_{1n}) \\
& =c(1\otimes e_{nn}).
\end{align*}
Hence $x(1\otimes\sum_{i=1}^n e_{ii})=c(1\otimes\sum_{i=1}^n e_{ii})$ for any $n\in\BN$.
Since $1\otimes\sum_{i=1}^n e_{ii}$ is strictly convergent to $1_{M(C\otimes \BK)}$
as $n\to\infty$, $x=c1$. Thus $(A\otimes \BK)' \cap M(C\otimes\BK)=\BC1$.
\newline
$(2)\Rightarrow (1)$: Let $x\in A' \cap M(C)$. Then $x\otimes 1_{M(\BK)}\in M(C\otimes\BK)$.
For any $a\in A$, $k\in \BK$,
$$
(a\otimes k)(x\otimes 1_{M(\BK)})=ax\otimes k =xa\otimes k=(x\otimes 1_{M(\BK)})(a\otimes k) .
$$
Thus $x\otimes 1_{M(\BK)}\in (A\otimes\BK)' \cap M(C\otimes\BK)=\BC1_{M(C\otimes\BK)}$.
Therefore, there is a $c\in \BC$
such that $x\otimes 1_{M(\BK)}=c1_{M(C\otimes\BK)}$.
Hence, $x=c1_{M(C\otimes\BK)}$. Thus
$A' \cap M(C)=\BC1$.
\end{proof}

Let $A$ and $B$ be $\sigma$-unital $C^*$-algebras and $A\subset C$ and $B\subset D$
inclusions of $C^*$-algebras with $\overline{AC}=C$ and $\overline{BD}=D$. Then $C$ and $D$ are
also $\sigma$-unital.
We suppose that $A\subset C$ and $B\subset D$ are strongly Morita equivalent.
Then in the same way as in the proof of \cite [Proposition 3.5]{Kodaka:Picard2} or Brown, Green and
Rieffel \cite[Proposition 3.1]{BGR:linking}, there is an isomorphism $\theta$ of $D\otimes\BK$ onto
$C\otimes\BK$ such that $\theta|_{B\otimes\BK}$ is an isomorphism of $B\otimes\BK$ onto $A\otimes\BK$.
Let $p=\underline{\theta}(1_{M(B)} \otimes e_{11})$. Then $p$ is a projection in
$M(A\otimes\BK)\subset M(C\otimes\BK)$ and $p$ is full in $A\otimes\BK$ and $C\otimes\BK$,
that is,
$$
\overline{(A\otimes\BK)p(A\otimes\BK)}=A\otimes\BK
$$ 
and
$$\overline{(C\otimes\BK)p(C\otimes\BK)}=C\otimes\BK .
$$
By the definitions of $\theta$ and $p$, the inclusion of $C^*$-algebras
$$
(1_{M(B)}\otimes e_{11})(B\otimes\BK)(1_{M(B)}\otimes e_{11})\subset (1_{M(B)}\otimes e_{11})
(D\otimes\BK)(1_{M(B)}\otimes e_{11})
$$
is isomorphic to the inclusion of $C^*$-algebras
$$
p(A\otimes\BK)p\subset p(C\otimes\BK)p
$$
as inclusions of $C^*$-algebras. We will show that if $B' \cap M(D)=\BC1$, then $A' \cap M(C)=\BC1$.

\begin{lemma}\label{lem:com2} With the above notation and assumptions,
$$
(A\otimes\BK)' \cap M(C\otimes\BK)
\cong ((A\otimes\BK)' \cap M(C\otimes\BK))p .
$$
\end{lemma}
\begin{proof}
Let $\pi$ be the map from $(A\otimes\BK)' \cap M(C\otimes\BK)$ to
$((A\otimes\BK)' \cap M(C\otimes \BK))p$
defined by $\pi(x)=xp$ for any $x\in (A\otimes\BK)' \cap M(C\otimes\BK)$.
Let $x\in (A\otimes\BK)' \cap M(C\otimes\BK)$.
Then for any $a\in A\otimes\BK$, $xa=ax$. Hence
$xa=ax$ for any $a\in M(A\otimes\BK)$. This implies that $\pi$ is a homomorphism of
$(A\otimes\BK)' \cap M(C\otimes\BK)$ to $((A\otimes\BK)' \cap M(C\otimes\BK))p$. It is clear that
$\pi$ is surjective. We show that $\pi$ is injective. We suppose that $\pi(x)=0$. Then $xp=0$.
For any $y, z\in A\otimes\BK$, $yxpz=0$. Since $x\in (A\otimes\BK)'$, $xypz=0$.
Since $\overline{(A\otimes\BK)p(A\otimes\BK)}=A\otimes\BK$,
$xa=0$ for any $a\in A\otimes\BK$. Let $c$ be any element in $C\otimes\BK$
and let $\{u_i \}$ be an approximate unit of $A\otimes\BK$.
Then $u_i c$ is norm convergent to $c$ since
$$
\overline{(A\otimes\BK)(C\otimes\BK)}=C\otimes\BK .
$$
Thus $xc=\lim_i xu_i c =0$ since $xu_i =0$ for any $i$. Since $x$ can be regarded as a double centralizer on $A$ by
multiplication by \cite [Proposition 3.12.3]{Pedersen:auto}, 
$x=0$. Since $\pi$ is injective, we obtain the conclusion.
\end{proof}

\begin{lemma}\label{lem:com3} With the above notation and assumptions, 
$$
((A\otimes\BK)' \cap M(C\otimes\BK))p\subset (p(A\otimes\BK)p)' \cap pM(C\otimes\BK)p .
$$
\end{lemma}
\begin{proof} Let $x\in (A\otimes\BK)' \cap M(C\otimes\BK)$. We show that
$$
xp\in (p(A\otimes\BK)p)' \cap pM(C\otimes\BK)p .
$$
We note that $xa=ax$ for any $a\in A\otimes\BK$.
Hence $xp=pxp$ since $p\in M(A\otimes\BK)$.
Thus $xp\in pM(C\otimes\BK)p$. Also, for any $a\in A\otimes\BK$,
$xp(pap)=xpap=(pap)xp$. Hence $xp\in(p(A\otimes\BK)p)'$. 
The claim clearly follows.
\end{proof}

\begin{lemma}\label{lem:com4} With the above notation and assumptions, if $B' \cap M(D)=\BC1$, then
$A' \cap M(C)=\BC1$.
\end{lemma}
\begin{proof} By the discussion before Lemma \ref{lem:com2},
$$
(1_{M(B)}\otimes e_{11})(B\otimes\BK)(1_{M(B)}\otimes e_{11})
\subset (1_{M(B)}\otimes e_{11})(D\otimes\BK)(1_{M(B)}\otimes e_{11})
$$
is isomorphic to
$(p(A\otimes\BK)p\subset p(C\otimes\BK)p$ as inclusions of $C^*$-algebras. 
Also,
$$
(1_{M(B)}\otimes e_{11})(B\otimes\BK)(1_{M(B)}\otimes e_{11})
\subset (1_{M(B)}\otimes e_{11})(D\otimes\BK)(1_{M(B)}\otimes e_{11})
$$
is isomorphic to
$B\subset D$ as inclusions of $C^*$-algebras. Hence since $B' \cap D=\BC 1$,
$(p(A\otimes\BK)p)' \cap p(C\otimes\BK)p=\BC p$.
Thus by Lemma \ref {lem:com3}
$((A\otimes\BK)' \cap M(C\otimes\BK))p=\BC p$.
It follows that $(A\otimes\BK)' \cap M(C\otimes\BK)=\BC 1$ by Lemma \ref {lem:com2}. 
Therefore, by Lemma \ref {lem:com1}, $A' \cap M(C)=\BC1$.
\end{proof}

\section{Free twisted actions of a countable discrete group}\label{sec:free}
We recall the following definitions (see \cite {Choda:free, Zarikian:expectation} and references therein).
Let $\alpha$ be an automorphism of a $C^*$-algebra $A$. 

\begin{Def}\label{def:free0} We say that $\alpha$ is {\sl inner} if there is a unitary
element $w\in M(A)$ such that $\alpha=\Ad(w)$. We say that $\alpha$ is {\sl outer} if $\alpha$ is not
inner.
\end{Def}

\begin{Def}\label{def:free1} We say that $\alpha$ is {\sl free} if $\alpha$ satisfies the following condition:
\newline
If $x\in M(A)$ satisfies that $xa=\alpha(a)x$ for all $a\in A$, then $x=0$.
\end{Def}

Let $(\alpha, w_{\alpha})$ be a twisted action of $G$ on a $C^*$-algebra $A$.

\begin{Def}\label{def:free3} We say that $(\alpha, w_{\alpha})$ is {\sl outer} if $\alpha_t$ is outer
for every $t\in G\setminus\{e\}$.
\end{Def}

\begin{Def}\label{def:free2} We say that $(\alpha, w_{\alpha})$ is {\sl free}
if the automorphism $\alpha_t$ is free for
every $t\in G\setminus\{e\}$.
\end{Def}

\begin{lemma}\label{lem:free9} Let $(\alpha, w_{\alpha})$ be a twisted action of $G$ on
a $C^*$-algebra $A$.
If $(\alpha, w_{\alpha})$ is free, then $(\alpha, w_{\alpha})$ is outer.
Furthermore, if $A' \cap M(A)=\BC1$, then the converse holds.
\end{lemma}
\begin{proof} We suppose first that $(\alpha, w_{\alpha})$ is free and that
$(\alpha, w_{\alpha})$ is not outer.
Then there are an element $t\in G\setminus\{e\}$ and a unitary element $w\in M(A)$ such that
$\alpha_t =\Ad(w)$. Thus for any $a\in A$, $\alpha_t (a)=waw^*$. Hence $wa=\alpha_t (a)w$
for all $a\in A$. Since $(\alpha, w_{\alpha})$ is free, $w=0$. This gives a contradiction.
Therefore, $(\alpha, w_{\alpha})$ is outer.
Next, we suppose that $A' \cap M(A)=\BC1$ and that $(\alpha, w_{\alpha})$ is outer.
Assume that $(\alpha, w_{\alpha})$ is not
free. Then there are an element $t\in G\setminus\{e\}$ and a non-zero element
$x\in M(A)$ such that $xa=\alpha_t (a)x$ for all $a\in A$. Hence $ax^* =x^* \alpha_t (a)$
for any $a\in A$. Thus $x^* xa=x^* \alpha_t (a)x=ax^* x$ for any $a\in A$. Hence $x^* x\in \BC1$.
Also, $\alpha_t (a)xx^* =xax^* =xx^* \alpha_t (a)$ for any $a\in A$. Thus $xx^* \in \BC1$.
A computation gives $(x^* x-xx^* )^* (x^* x-xx^*)=0$, i.e., $x$ is normal, so we get that
there exists some positive $c\in \BR$ such that $x^* x=xx^* =c1$.
Let $w=\frac{1}{\sqrt{c}}x$. Then $w$ is a unitary element in $M(A)$
such that $\alpha_t (a)=waw^*$ for all $a\in A$. This gives a contradiction. Therefore, $(\alpha, w_{\alpha})$ is
free.
\end{proof}

Let $(\alpha, w_{\alpha})$ be a twisted action of $G$ on a unital $C^*$algebra $A$.
Let $E^A$ be the faithful canonical conditional expectation from $A\rtimes_{\alpha, w_{\alpha}, r}G$
onto $A$ defined in Section \ref {sec:pre}. In the same way as
Zarikian \cite [Theorem 3.1.2]{Zarikian:expectation}, we obtain the following proposition.

\begin{prop}\label{prop:free4} With the above assumption,
the following conditions are equivalent:
\newline
$(1)$ The conditional expectation $E^A$ is unique,
\newline
$(2)$ $A' \cap (A\rtimes_{\alpha, w_{\alpha}, r}G)=A' \cap A$,
\newline
$(3)$ the twisted action $(\alpha, w_{\alpha})$ is free.
\end{prop}
\begin{proof} $(1)\Rightarrow(2)$: This is immediate by \cite [Lemma 3.1.1]{Zarikian:expectation}.
\newline
$(2)\Rightarrow (3)$: Let $t\in G\setminus\{e\}$ and $x\in A$. We suppose that $xa=\alpha_t (a)x$
for all $a\in A$. Then since $\alpha_t (a)=\delta_t a\delta_t^*$, $\delta_t^* xa=a\delta_t^* x$ for all
$a\in A$. Hence $\delta_t^* x\in A' \cap(A\rtimes_{\alpha, w_{\alpha}, r}G)$.
By the definition of $E^A$, $E^A (\delta_t^* x)=0$. On the other hand, since $A' \cap
(A\rtimes_{\alpha, w_{\alpha}, r}G)=A' \cap A$, $E^A (\delta_t^* x)=\delta_t^* x$.
Hence $x=0$.
\newline
$(3)\Rightarrow (1)$: We suppose that $(\alpha, w_{\alpha})$ is free. Let $F$ be a conditional
expectation from $A\rtimes_{\alpha, w_{\alpha}, r}G$ onto $A$. Let $t\in G\setminus\{e\}$.
For any $a\in A$,
$$
F(\delta_t )a=F(\delta_t a)=F(\alpha_t (a)\delta_t )=\alpha_t (a)F(\delta_t )
$$
since $\alpha_t (a) =\delta_t a\delta_t^*$. Since $(\alpha, w_{\alpha})$ is free, $F(\delta_t )=0$.
Hence $F(a\delta_t )=aF(\delta_t )=0=E^A (a\delta_t )$ for all $a\in A$, $t\in G\setminus\{e\}$, and
it follows that $F=E^A$.
\end{proof}

\begin{cor}\label{cor:free5} Let $(\alpha, w_{\alpha})$ be a twisted action of $G$ on a $C^*$-algebra
$A$. Then the following conditions are equivalent:
\newline
$(1)$ $A' \cap M(A\rtimes_{\alpha, w_{\alpha}, r}G)=\BC1$,
\newline
$(2)$ $A' \cap M(A)=\BC1$ and $(\alpha, w_{\alpha})$ is free,
\newline
$(3)$ $A' \cap M(A)=\BC1$ and $(\alpha, w_{\alpha})$ is outer.
\end{cor}
\begin{proof} $(1)\Rightarrow (2)$: It is clear that $A' \cap M(A)=\BC1$. We show that
$(\alpha, w_{\alpha})$ is free. Since $M(A)\rtimes_{\underline{\alpha}, w_{\alpha}, r}G\subset
M(A\rtimes_{\alpha, w_{\alpha}, r}G)$, 
$$
M(A)' \cap(M(A)\rtimes_{\underline{\alpha}, w_{\alpha}, r}G)
\subset A' \cap M(A\rtimes_{\alpha, w_{\alpha}, r}G)=\BC1 .
$$
Hence by Proposition \ref {prop:free4}, $(\underline{\alpha}, w_{\alpha})$ is free. Let $t\in G\setminus\{e\}$
and $x\in M(A)$. We suppose that $xa=\alpha_t (a)x$ for all $a\in A$. Then since $\underline{\alpha_t}$
is strictly continuous, $xa=\underline{\alpha_t}(a)x$ for all $a\in M(A)$. Hence
$x=0$ since $(\underline{\alpha}, w_{\alpha})$ is free. Thus $(\alpha, w_{\alpha})$ is free.
\newline
$(2)\Rightarrow (1)$: Let $x\in A' \cap M(A\rtimes_{\alpha, w_{\alpha}, r}G)$. For any $t\in G$
let $x_t$ be the Fourier coefficient of $x$ at $t\in G\setminus \{e\}$ defined in
Section \ref{sec:pre}. Then for any $a\in A$,
\begin{align*}
x_t a & =\lim_i E^A (xu_i \delta_t^* )a=\lim_i E^A (xu_i \delta_t^* a) 
=\lim_i E^A (xu_i \alpha_t^{-1} (a)\delta_t^* ) \\
& =E^A (x\alpha_t^{-1} (a)\delta_t^* )=E^A (\alpha_t^{-1} (a)x\delta_t^* )=\lim_i  E^A (\alpha_t^{-1} (a)xu_i \delta_t^* ) \\
& =\alpha_t^{-1} (a)\lim_i E^A (xu_i \delta_t^* )=\alpha_t^{-1} (a)x_t .
\end{align*}
Since $(\alpha, w_{\alpha})$ is free $x_t =0$ for every $t\in G\setminus\{e\}$.
Thus, setting $y=x-x_e \in M(A\rtimes_{\alpha, w_{\alpha}, r}G)$,
we get that $y_t =0$ for all $t\in G$. Lemma \ref{lem:pre7} gives that $y=0$.
Hence, $x=x_e \in A' \cap M(A)=\BC 1$. This shows that (1) holds.
\newline
The equivalence $(2)\Leftrightarrow(3)$ follows from \ref{lem:free9}.
\end{proof}

\begin{lemma}\label{lem:free8} Let $(\alpha, w_{\alpha})$ be a free twisted action of $G$ on
a $C^*$-algebra $A$. Then $(\alpha\otimes\id, w_{\alpha}\otimes 1)$ is a free twisted action of
$G$ on $A\otimes\BK$.
\end{lemma}
\begin{proof} Let $t\in G\setminus\{e\}$.
Let $x$ be an element in $M(A\otimes\BK)$ satisfying that
$$
xy=(\alpha_t \otimes\id)(y)x
$$
for all $y\in A\otimes\BK$. This implies that
$$
xy=(\underline{\alpha_t}\otimes\id)(y)x
$$
for all $y\in M(A)\otimes\BK$. Hence we get that
$$
x(1\otimes e_{ii})=(1\otimes e_{ii})x
$$
for all $i\in \BN$. Also, we have that
$$
x(a\otimes e_{ii})=(\alpha_t (a)\otimes e_{ii})x
$$
for any $a\in A$, $i\in \BN$. Hence
$$
(1\otimes e_{ii})x(1\otimes e_{ii})(a\otimes e_{ii})
=(\alpha_t (a)\otimes e_{ii})(1\otimes e_{ii})x(1\otimes e_{ii}) 
$$
for all $a\in A$, $i\in \BN$. Since $(\alpha, w_{\alpha})$ is a free twisted action on $A$ and 
we can identify $A$ with
$(1\otimes e_{ii})(A\otimes \BK)(1\otimes e_{ii})$,
$$
(1\otimes e_{ii})x(1\otimes e_{ii})=0
$$
for all $i\in \BN$. Thus, since $x(1\otimes e_{ii})=0$ for any $i\in \BN$, $x=0$.
\end{proof}

\section{Strong Morita equivalence for inclusions of $C^*$-algebras}\label{sec:inclusion}
Let $(\alpha, w_{\alpha})$ and $(\beta, w_{\beta})$ be twisted actions of $G$ on
$\sigma$-unital $C^*$-algebras $A$ and $B$
and let $A\rtimes_{\alpha, w_{\alpha}, r}G$ and $B\rtimes_{\beta, w_{\beta}, r}G$ denote
the twisted reduced crossed products of $A$
and $B$ by $(\alpha, w_{\alpha})$ and $(\beta, w_{\beta})$, respectively. Then we obtain the 
inclusions of $C^*$-algebras $A\subset A\rtimes_{\alpha, w_{\alpha}, r}G$ and
$B\subset B\rtimes_{\beta, w_{\beta}, r}G$.
\par
We suppose that $A\subset A\rtimes_{\alpha, w_{\alpha}, r}G$ and
$B\subset B\rtimes_{\beta, w_{\beta}, r}G$ are
strongly Morita equivalent. We regard $\BK$ as the trivial $\BK-\BK$-equivalence bimodule.
Then the inclusions of $C^*$algebras $A\otimes\BK\subset (A\rtimes_{\alpha, w_{\alpha}, r}G)\otimes\BK$
and $B\otimes\BK\subset (B\rtimes_{\beta, w_{\beta}, r}G)\otimes\BK$ are also strongly Morita equivalent.
Hence in the same way as in the proof of \cite [Proposition 3.5]{Kodaka:Picard2} or
Brown, Green and Rieffel
\cite [Proposition 3.1]{BGR:linking}, there is an isomorphism $\theta$ of
$(A\rtimes_{\alpha, w_{\alpha}, r}G)\otimes\BK$ onto
$(B\rtimes_{\beta, w_{\beta}, r}G)\otimes\BK$ such that $\theta|_{A\otimes \BK}$ is an isomorphism of
$A\otimes\BK$ onto $B\otimes\BK$. Also, as explained in section \ref{sec:pre},
the inclusion of $C^*$-algebras
$$
A\otimes\BK\subset (A\rtimes_{\alpha, w_{\alpha}, r}G)\otimes\BK
$$
is isomorphic to
$$
A\otimes \BK\subset (A\otimes\BK)\rtimes_{\alpha\otimes\id, w_{\alpha}\otimes 1, r}G
$$
as inclusions of $C^*$-algebras, and
$$
B\otimes\BK\subset (B\rtimes_{\beta, w_{\beta}, r}G)\otimes \BK
$$
is isomorphic to
$$
B\otimes\BK\subset (B\otimes\BK)\rtimes_{\beta\otimes\id, w_{\beta}\otimes 1, r}G
$$
as inclusions of $C^*$-algebras. 
Thus $\theta$ can be regarded as an isomorphism of
$(A\otimes\BK)\rtimes_{\alpha\otimes\id, w_{\alpha}\otimes 1, r}G$ onto
$(B\otimes\BK)\rtimes_{\beta\otimes\id, w_{\beta}\otimes 1, r}G$ such that
$\theta|_{A\otimes \BK}$ is an isomorphism of
$A\otimes\BK$ onto $B\otimes\BK$. Set $\theta_r =\theta|_{A\otimes\BK}$,
$\gamma_t =\theta_r \circ(\alpha_t \otimes\id) \circ\theta_r^{-1}$ and
$w_{\gamma}(t, s)=\underline{\theta_r}(w_{\alpha}(t, s)\otimes 1)$ for all $t, s\in G$.

\begin{lemma}\label{lem:strong0} With the above notation and assumptions,
$(\alpha\otimes\id, w_{\alpha} \otimes 1)$ and
$(\gamma, w_{\gamma})$ are strongly Morita equivalent.
\end{lemma}
\begin{proof} Let $\theta_r$ be the isomorphism of $A\otimes\BK$ onto $B\otimes\BK$
defined as above. Let $X_{\theta_r^{-1}}$ be the $A\otimes\BK-B\otimes\BK$-equivalence bimodule
obtained in the following way: Set $X_{\theta_r^{-1}}=A\otimes\BK$ as a vector space over $\BC$.
For any $a\in A\otimes\BK$, $b\in B\otimes\BK$, $x, y\in X_{\theta_r^{-1}}$, define
\begin{align*}
a\cdot x & =ax, \quad x\cdot b=x\theta_r^{-1}(b) , \\
{}_{A\otimes\BK} \la x, y \ra& =xy^* , \quad \la x, y \ra_{B\otimes\BK}=\theta_r (x^* y) .
\end{align*}
By easy computations, $X_{\theta_r^{-1}}$ is an $A\otimes\BK-B\otimes\BK$-equivalence bimodule.
For any $t\in G$, let $\lambda_t \in \Aut(X_{\theta_r^{-1}})$ defined by
$$
\lambda_t (x)=(\alpha_t \otimes\id)(x)
$$
for all $x\in X_{\theta_r^{-1}}$. Then
\begin{align*}
{}_{A\otimes\BK} \la \lambda_t (x) \, , \, \lambda_t (y) \ra & =(\alpha_t \otimes\id)(xy^* )
=(\alpha_t \otimes\id)({}_{A\otimes\BK} \la x , y \ra) , \\
\la \lambda_t (x) \, , \, \lambda_t (y) \ra_{B\otimes\BK} & =\theta_r ((\alpha_t \otimes\id)(x^* y))
=\gamma_t (\theta_r (x^* y ))=\gamma_t (\la x, y \ra_{B\otimes\BK})
\end{align*}
for all $x, y\in X_{\theta_r^{-1}}$, $t\in G$. Therefore, we obtain the conclusion.
\end{proof}
Since the systems $(A\otimes\BK, G, \alpha\otimes\id, w_{\alpha}\otimes 1)$
and $(B\otimes\BK, G, \gamma, w_{\gamma})$ are conjugate,
there is an isomorphism $\pi_{\theta}$
of $(A\otimes\BK)\rtimes_{\alpha\otimes\id, w_{\alpha}\otimes 1, r}G$ onto
$(B\otimes\BK)\rtimes_{\gamma, w_{\gamma}, r}G$
such that $\pi_{\theta}|_{A\otimes\BK}=\theta_r $.
Then we have the following commutative diagram:
$$
\begin{matrix}
(B\otimes\BK)\!\rtimes_{\beta\otimes\id, w_{\beta}\otimes 1, r}G & \overset{\theta}\leftarrow &
(A\otimes\BK)\!\rtimes_{\alpha\otimes\id, w_{\alpha}\otimes 1, r}G & \overset{\pi_{\theta}}\rightarrow &
(B\otimes\BK)\!\rtimes_{\gamma, w_{\gamma}, r}G \\
\uparrow & \quad & \uparrow \quad & \quad & \uparrow \\
B\otimes\BK & \overset{\theta_r}\leftarrow &  A\otimes\BK & \overset{\theta_r}\rightarrow & \quad 
B\otimes\BK \quad .
\end{matrix} 
$$
Let $\rho$ be the isomorphism of
$(B\otimes\BK)\rtimes_{\beta\otimes \id, w_{\beta}\otimes 1, r}G$ onto
$(B\otimes\BK)\rtimes_{\gamma, w_{\gamma}, r}G$ defined by
$\rho=\pi_{\theta}\circ\theta^{-1}$.
Then
%for any $b\in B\otimes\BK$
%$$
%\rho(b)=(\pi_{\theta}\circ\theta^{-1})(b)=\theta_r (\theta_r^{-1} (b))=b .
%$$
%Thus
$\rho|_{B\otimes\BK}=\id$ on $B\otimes\BK$. For any $t\in G$,
let $u_t^{\beta\otimes\id}$ be the canonical unitary element in
$M((B\otimes\BK)\rtimes_{\beta\otimes\id, w_{\beta}\otimes 1, r}G)$
implementing $\beta_t \otimes\id$. Set $v_t =\underline{\rho}(u_t^{\beta\otimes\id})$ for each $t\in G$.
Then for any $b\in B\otimes\BK$,
$$
v_t b v_t^* =\underline{\rho}(u_t^{\beta\otimes\id})\rho(b)\underline{\rho}(u_t^{\beta \otimes\id \,*})
=\rho((\beta_t \otimes\id)(b))=(\beta_t \otimes\id)(b) .
$$
Hence $v_t b=(\beta_t \otimes\id)(b)v_t$ for all $b\in B\otimes\BK$, $t\in G$.
\par
Let $E^{B\otimes\BK}$ be the canonical conditional expectation from
$(B\otimes\BK)$
$\rtimes_{\gamma, w_{\gamma}, r}G$ onto $B\otimes\BK$ and let
$\{U_i \}_{i\in I}$ be an approximate unit of $B\otimes\BK$. Let $\{b_s^t\}_{s\in G}$
be the Fourier coefficients of $v_t$ in $M((B\otimes\BK)\rtimes_{\gamma, w_{\gamma}, r}G)$
with respect to $\{U_i \}_{i\in I}$.

\begin{lemma}\label{lem:strong1} With the above notation and assumptions, for any $t, s\in G$, $b\in B\otimes\BK$,
we have
$$
b_s^t b=((\beta_t \otimes\id)\circ\gamma_s^{-1})(b)b_s^t .  \quad (*)
$$
\end{lemma}
\begin{proof}
Let $u_s^{\gamma}$ be the canonical unitary element in
$M((B\otimes\BK)\rtimes_{\gamma, w_{\gamma},  r}G)$ implementing
$\gamma_s$ for any $s\in G$. Let $t\in G$ and $b\in B\otimes\BK$. Then
since $||bU_i -U_i b||\to 0 \, (i\to \infty)$,
the Fourier coefficient of $v_t b$ at $s\in G$ is given by
\begin{align*}
\lim_i E^{B\otimes\BK}(v_t bU_i u_s^{\gamma *}) & =\lim_i E^{B\otimes\BK}(v_t U_i bu_s^{\gamma *})
=\lim_i E^{B\otimes\BK}(v_t U_i u_s^{\gamma * } \gamma_s (b)) \\
& =\lim_i E^{B\otimes\BK}(v_t U_i u_s^{\gamma* })\gamma_s (b)
=b_s^t \gamma_s (b) .
\end{align*}
Moreover the Fourier coefficient of $(\beta_t \otimes\id)(b)v_t$ at $s\in G$ is given by
\begin{align*}
\lim_i E^{B\otimes\BK}((\beta_t \otimes\id)(b)v_t U_i u_s^{\gamma *}) & =(\beta_t \otimes\id)(b)\lim_i
E^{B\otimes\BK}(v_t U_i u_s^{\gamma *}) \\
& =(\beta_t \otimes\id)(b)b_s^t .
\end{align*}
Since $v_t b=(\beta_t \otimes\id)(b)v_t$, we get that
$$
b_s^t \gamma_s (b)=(\beta_t \otimes\id)(b)b_s^t
$$
for all $t, s\in G$, $b\in B\otimes\BK$. Since $b$ is an arbitrary element in $B\otimes\BK$,
replacing $b$ by $\gamma_s^{-1}(b)$, we obtain that
$$
b_s^t b=(\beta_t \otimes\id)(\gamma_s^{-1}(b))b_s^t 
$$
for all $t, s\in G$, $b\in B\otimes\BK$, as desired.
\end{proof}

Assume now that the actions $(\alpha\otimes\id, w_{\alpha}\otimes 1)$ and
$(\beta\otimes\id, w_{\beta}\otimes 1)$ are free on $A\otimes\BK$ and
$B\otimes\BK$, respectively and that $(B\otimes\BK)'\cap M(B\otimes\BK)=\BC1_{M(B\otimes\BK)}$.
Let $t$ be any element in $G$. Since $v_t \ne 0$, there is an element $s_0 \in G$ such that
$b_{s_0}^t \ne 0$. Since $s_0 \in G$ is depending on $t\in G$, we denote it by $\phi(t)$.
Hence $b_{\phi(t)}^t \ne 0$, and by Lemma \ref {lem:strong1} for any $b\in B\otimes\BK$, we have
$$
b_{\phi(t)}^t b =((\beta_t \otimes\id)\circ\gamma_{\phi(t)}^{-1})(b)b_{\phi(t)}^t .
$$
Since $(B\otimes\BK)' \cap M(B\otimes\BK)=\BC1$, by the proof of Lemma \ref {lem:free9},
there is a unitary element $w_t \in M(B\otimes\BK)$ such that
$$
((\beta_t \otimes\id)\circ\gamma_{\phi(t)}^{-1})(b)=w_t bw_t^*
$$
for all $b\in B\otimes\BK$. That is,
$$
(\beta_t \otimes\id)(b)=w_t \gamma_{\phi(t)}(b)w_t^* 
$$
for all $b\in B\otimes\BK$. Thus for any $s\in G$, $b\in B\otimes\BK$,
$$
(\beta_t \otimes\id)(\gamma_{s^{-1}}(b))=w_t \gamma_{\phi(t)}(\gamma_{s^{-1}}(b))w_t^* .
$$
By Equation $(*)$ in Lemma \ref{lem:strong1} and the above equation,
\begin{align*}
b_s^t b & =(\beta_t \otimes\id)(\gamma_s^{-1}(b))b_s^t \\
& =((\beta_t \otimes\id)\circ\Ad(w_{\gamma}(s^{-1}, s)^* \circ\gamma_{s^{-1}})(b)b_s^t \\
& =[\Ad((\underline{\beta_t} \otimes\id)(w_{\gamma}(s^{-1}, s)^* ))
\circ(\beta_t \otimes\id)\circ\gamma_{s^{-1}}](b)b_s^t \\
& =[\Ad((\underline{\beta_t} \otimes\id)(w_{\gamma}(s^{-1}, s)^* ))
\circ\Ad(w_t )\circ\gamma_{\phi(t)}\circ\gamma_{s^{-1}}](b)b_s^t \\
& =[\Ad((\underline{\beta_t} \otimes\id)(w_{\gamma}(s^{-1}, s)^* )w_t )
\circ\Ad(w_{\gamma}(\phi(t), s^{-1})\circ\gamma_{\phi(t)s^{-1}}](b)b_s^t \\
& =[\Ad((\underline{\beta_t} \otimes\id)(w_{\gamma}(s^{-1}, s)^*) w_t w_{\gamma}(\phi(t), s^{-1}))
\circ\gamma_{\phi(t)s^{-1}}](b)b_s^t .
\end{align*}
Hence
\begin{align*}
w_{\gamma}(\phi(t), s^{-1})^* w_t^* (\underline{\beta_t} \otimes\id) & (w_{\gamma}(s^{-1}, s))b_s^t b \\
& =\gamma_{\phi(t)s^{-1}}(b)w_{\gamma}(\phi(t), s^{-1})^* w_t^* (\underline{\beta_t} \otimes\id)
(w_{\gamma}(s^{-1}, s))b_s^t
\end{align*}
for any $s\in G$, $b\in B\otimes\BK$. Since $(\alpha\otimes\id, w_{\alpha}\otimes 1)$ is free, so
is $(\gamma, w_{\gamma})$. Hence the automorphism $\gamma_{\phi(t)s^{-1}}$ is free if $s\ne\phi(t)$.
Thus $b_s^t =0$ for any $s\in G$ with $s\ne \phi(t)$.
Therefore, for any $t\in G$, there is a unique $\phi(t)\in G$ such that
$$
b_{\phi(t)}^t \ne 0 , \quad \underline{\rho}(u_t^{\beta\otimes\id})=v_t =b_{\phi(t)}^t u_{\phi(t)}^{\gamma} .
$$
By the above discussions, we can regard $\phi$ as a map on $G$.
We show that $\phi$ is an automorphism of $G$.

\begin{lemma}\label{lem:strong2} With above notation and assumptions, $\phi$ is an automorphism of $G$.
\end{lemma}
\begin{proof} Let $t, s\in G$. Then
$$
v_t v_s =\underline{\rho}(u_t^{\beta\otimes\id}u_s^{\beta\otimes\id})
=\underline{\rho}((w_{\beta}(t, s)\otimes 1)u_{ts}^{\beta\otimes\id}) 
=\underline{\rho}(w_{\beta}(t, s)\otimes 1)v_{ts} .
$$
Since $\rho(b)=b$ for any $b\in B\otimes\BK$,
$\underline{\rho}(w_{\beta}(t, s)\otimes 1)=w_{\beta}(t, s)\otimes 1$.
Hence $v_t v_s =(w_{\beta}(t, s)\otimes 1)v_{ts}$. Then the Fourier coefficient of $v_t v_s$ at
$r\in G$ can be computed as follows: Since $v_t =b_{\phi(t)}^t u_{\phi(t)}^{\gamma}$
and $v_s =b_{\phi(s)}^s u_{\phi(s)}^{\gamma}$, we get that
\begin{align*}
\lim_i E^{B\otimes\BK}(v_t v_s U_i u_r^{\gamma *}) 
& =\lim_i E^{B\otimes\BK}(b_{\phi(t)}^t u_{\phi(t)}^{\gamma}b_{\phi(s)}^s u_{\phi(s)}^{\gamma}U_i u_r^{\gamma *}) \\
& =\lim_i E^{B\otimes\BK}(b_{\phi(t)}^t \underline{\gamma}_{\phi(t)}(b_{\phi(s)}^s )u_{\phi(t)}^{\gamma}
u_{\phi(t)}^{\gamma}U_i u_r^{\gamma *}) \\
& =\lim_i E^{B\otimes\BK}(b_{\phi(t)}^t \underline{\gamma}_{\phi(t)}(b_{\phi(s)}^s )w_{\gamma}(\phi(t), \phi(s))
u_{\phi(t)\phi(s)}^{\gamma}U_i u_r^{\gamma *}) \\
& =\begin{cases} b_{\phi(t)}^t \underline{\gamma}_{\phi(t)}(b_{\phi(s)}^s )w_{\gamma}(\phi(t), \phi(s)) & \text{if
$r=\phi(t)\phi(s)$} \\
0 & \text{if $r\ne \phi(t)\phi(s)$.} \end{cases}
\end{align*}
On the other hand, the Fourier coefficient of $(w_{\beta}(t, s)\otimes 1)v_{ts}$ at $r\in G$ can be computed as follows:
Since $v_{ts}=b_{\phi(ts)}^{ts}u_{\phi(ts)}^{\gamma *}$, we get that
\begin{align*}
\lim_i E^{B\otimes\BK}((w_{\beta}(t, s)\otimes 1)v_{ts}U_i u_r^{\gamma *}) & =
\lim_i E^{B\otimes\BK}((w_{\beta}(t, s)\otimes 1)b_{\phi(ts)}^{ts}u_{\phi(ts)}U_i u_r^{\gamma *}) \\
& =\begin{cases} (w_{\beta}(t, s)\otimes 1)b_{\phi(ts)}^{ts} & \text{if $r=\phi(ts)$} \\
0 & \text{if $r\ne \phi(ts)$}. \end{cases}
\end{align*}
As $v_t v_s =(w_{\beta}(t, s)\otimes 1)v_{ts}$,  we obtain that
\begin{align*}
& u_{\phi(t)\phi(s)}^{\gamma}=u_{\phi(ts)}^{\gamma} , \\
& b_{\phi(t)}^t \underline{\gamma}_{\phi(t)}(b_{\phi(s)}^s )w_{\gamma}(\phi(t), \phi(s))
=(w_{\beta}(t, s)\otimes 1)b_{\phi(ts)}^{ts} \quad  (**)
\end{align*}
for all $t, s\in G$. Hence $\phi(t)\phi(s)=\phi(ts)$ for all $t, s\in G$, i.e., $\phi$ is a homomorphism
of $G$ to $G$. We show that $\phi$ is injective. Let $t\in G\setminus\{e\}$ and suppose that $\phi(t)=e$.
By the definition of $\phi$, $\underline{\rho}(u_t^{\beta\otimes\id})=v_t =b_e^t u_e^{\gamma}=b_e^t$.
By Equation $(*)$ in Lemma \ref{lem:strong1},
$$
b_e^t b=(\beta_t \otimes\id)(b)b_e^t
$$
for all $b\in B\otimes\BK$. Since $\beta_t \otimes\id$ is a free automorphism of $B\otimes\BK$,
$b_e^t =0$. Hence $v_t =0$.
This is a contradiction. Thus $\phi(t)\ne e$, that is, $\phi$ is injective.
Next we show that $\phi$ is surjective. Since $\rho$ is an isomorphism
of $(B\otimes\BK)\rtimes_{\beta\otimes\id, w_{\beta}\otimes 1, r}G$ onto
$(B\otimes\BK)\rtimes_{\gamma, w_{\gamma}, r}G$, $(B\otimes\BK)\rtimes_{\gamma, w_{\gamma}, r}G$
is the closed linear span of the set
$$
\{bv_t \, | \, b\in B\otimes\BK, \, t\in G \}
=\{bb_{\phi(t)}^t u_{\phi(t)}^{\gamma} \, | \, b\in B\otimes\BK \, , \, t\in  G \}
$$
Also,
$$
\{bb_{\phi(t)}^t u_{\phi(t)}^{\gamma} \, | \, b\in B\otimes\BK \, , \, t\in  G \}
=\{bu_{\phi(t)}^{\gamma} \, | \, b\in B\otimes\BK \, , \, t\in G \}
$$
since $b_{\phi(t)}^t$ is a unitary element in $M(B\otimes\BK)$.
Assume that $\phi$ is not surjective. Then there is some element $t_0 \in G$ such that
$t_0 \notin\phi(G)$. For any $b\in B\otimes\BK$ and $t\in G$, the Fourier coefficient of
$bu_{\phi(t)}^{\gamma}$ at $t_0 \in G$ is can be computed as follows:
$$
\lim_i E^{B\otimes\BK}(bu_{\phi(t)}^{\gamma}U_i u_{t_0}^{\gamma \, *} )=
\lim_i E^{B\otimes\BK}(b\gamma_{\phi(t)}(U_i )u_{\phi(t)}^{\gamma}u_{t_0 }^{\gamma \, *})=0 .
$$
Thus, for any $x\in (B\otimes\BK)\rtimes_{\gamma, w_{\gamma}, r}G$, the Fourier coefficient of
$x$ at $t_0 \in G$ is equal to $0$. On the other hand, for an element $b\in B\otimes \BK$ with $b\ne 0$,
$bu_{t_0}^{\gamma}\in (B\otimes\BK)\rtimes_{\gamma, w_{\gamma}, r}G$ and the Fourier coefficient of
$bu_{t_0}^{\gamma}$ at $t_0 \in G$ is equal to $b$. This is a contradiction.
Hence $\phi$ is surjective.
Therefore, we can see that $\phi$ is an automorphism of $G$.
\end{proof}

\begin{lemma}\label{lem:strong3} With the above notation and assumptions,
$(\beta\otimes\id, w_{\beta}\otimes 1)$ and
$(\gamma^{\phi}, w_{\gamma}^{\phi})$ are exterior equivalent, where $(\gamma^{\phi}, w_{\gamma}^{\phi})$ 
is the twisted action of $G$
on $B\otimes\BK$ defined by $\gamma_t^{\phi}(b)=\gamma_{\phi(t)}(b)$ and
$w_{\gamma}^{\phi}(t, s)=w_{\gamma}(\phi(t), \phi(s))$ for
all $t, s\in G$, $b\in B\otimes\BK$.
\end{lemma}
\begin{proof} By the discussion before Lemma \ref{lem:strong2}, $v_t =b_{\phi(t)}^t u_{\phi(t)}^{\gamma}$
for any $t\in G$. Since $v_t$ and $u_{\phi(t)}^{\gamma}$ are unitary elements in $M(B\otimes\BK)$,
$b_{\phi(t)}^t$
is a unitary element in $M(B\otimes\BK)$ for any $t\in G$. Also, for any $b\in B\otimes\BK$,
$t\in G$
\begin{align*}
(\beta_t \otimes\id)(b) & =v_t bv_t^* =b_{\phi(t)}^t u_{\phi(t)}^{\gamma}bu_{\phi(t)}^{\gamma \, *}
b_{\phi(t)}^{t \, *} =b_{\phi(t)}^t \gamma_{\phi(t)}(b)b_{\phi(t)}^{t \, *} \\
& =(\Ad(b_{\phi(t)}^t )\circ \gamma_{\phi(t)})(b)
\end{align*}
Moreover, Equation ($**$) in the proof of Lemma \ref {lem:strong2} says that
$$
b_{\phi(t)}^t \underline{\gamma}_{\phi(t)}(b_{\phi(s)}^s)w_{\gamma}(\phi(t), \phi(s))
=(w_{\beta}(t, s)\otimes 1 )b_{\phi(ts)}^{ts}
$$
for all $t, s\in G$. Therefore, $(\beta, w_{\beta})$ and $(\gamma^{\phi}, w_{\gamma}^{\phi})$
are exterior equivalent.
\end{proof}

We can now state the main theorem, which shows that the converse of Proposition
\ref{prop:pre4} holds under the irreduciblity condition described in Corollary \ref{cor:free5}.

\begin{thm}\label{thm:strong4} Let $(\alpha, w_{\alpha})$ and $(\beta, w_{\beta})$ be
twisted actions of a countable discrete group $G$
on $\sigma$-unital $C^*$-algebras $A$ and $B$, and let $A\rtimes_{\alpha, w_{\alpha}, r}G$ and
$B\rtimes_{\beta, w_{\beta}, r}G$ denote
the twisted reduced crossed products of $A$ and $B$ by $(\alpha, w_{\alpha})$ and
$(\beta, w_{\beta})$, respectively.
We suppose that
$A' \cap M(A\rtimes_{\alpha, w_{\alpha}, r}G)=\BC1$. If the inclusions of $C^*$-algebras
$A\subset A\rtimes_{\alpha, w_{\alpha}, r}G$ and $B\subset B\rtimes_{\beta, w_{\beta}, r}G$
are strongly Morita
equivalent, then there is an automorphism $\phi$ of $G$ such that $(\alpha^{\phi}, w_{\alpha}^{\phi})$
and $(\beta, w_{\beta})$ are strongly
Morita equivalent, where $(\alpha^{\phi}, w_{\alpha}^{\phi})$ is the twisted action of $G$ on $A$
defined by $\alpha_t^{\phi}(a)
=\alpha_{\phi(t)}(a)$ and $w_{\alpha}^{\phi}(t, s)=w_{\alpha}(\phi(t), \phi(s))$ for all $t, s\in G$, $a\in A$.
\end{thm}
\begin{proof} Since $A' \cap M(A\rtimes_{\alpha, w_{\alpha}, r}G)=\BC1$,
by Lemma \ref {lem:com4},
$B' \cap M(B\rtimes_{\beta, w_{\beta}, r}G)=\BC1$.
Also, by Corollary \ref{cor:free5}, $A' \cap M(A)=\BC1$, $B' \cap M(B)=\BC 1$ and $(\alpha, w_{\alpha})$,
$(\beta, w_{\beta})$ are free. Thus by Lemma \ref{lem:com1},
$(B\otimes\BK)' \cap M(B\otimes\BK)=\BC1_{M(B\otimes\BK)}$.
Furthermore, by Lemma \ref {lem:free8},
$(\alpha\otimes\id, w_{\alpha}\otimes 1)$ and $(\beta\otimes\id, w_{\beta}\otimes 1)$ are free.
On the other hand,
since $A\subset A\rtimes_{\alpha, w_{\alpha}, r}G$ and $B\subset B\rtimes_{\beta, w_{\beta}, r}G$
are strongly Morita equivalent,
by the discussions before Lemma \ref{lem:strong1}, there is an isomorphism $\theta$ of
$(A\otimes\BK)\rtimes_{\alpha\otimes\id, w_{\alpha}\otimes 1, r}G$ onto
$(B\otimes\BK)\rtimes_{\beta\otimes\id, w_{\beta}\otimes 1, r}G$ such
that $\theta|_{A\otimes\BK}$ is an isomorphism of $A\otimes\BK$ onto $B\otimes\BK$.
Set $\theta_r =\theta|_{A\otimes\BK}$ and $\gamma_t =\theta_r \circ (\alpha_t \otimes\id)\circ\theta_r^{-1}$,
$w_{\gamma}(t, s)=\underline{\theta_r}(w_{\alpha}(t, s)\otimes 1)$ for all $t, s\in G$.
Then by Lemmas \ref{lem:strong2} and \ref{lem:strong3}, there is an automorphism $\phi$ of
$G$ such that $(\beta\otimes\id, w_{\beta}\otimes 1)$ and
$(\gamma^{\phi}, w_{\gamma}^{\phi})$ are exterior
equivalent, hence strongly Morita equivalent,
where $(\gamma^{\phi}, w_{\gamma}^{\phi})$ is the twisted action of $G$ on $B\otimes\BK$ defined by
$\gamma_t^{\phi}(b)=\gamma_{\phi(t)}(b)$ and $w_{\gamma}^{\phi}(t, s)=w_{\gamma}(\phi(t), \phi(s))$
for all $t, s\in G$ and $b\in B\otimes\BK$. We note that $(\gamma, w_{\gamma})$ and
$(\alpha\otimes\id, w_{\alpha}\otimes 1)$ are strongly Morita equivalent by Lemma \ref{lem:strong0}.
It follows that $(\gamma^{\phi}, w_{\gamma}^{\phi})$ and $(\alpha^{\phi}\otimes\id, w_{\alpha}^{\phi} \otimes 1)$
are strongly Morita equivalent. 
Hence we can conclude that $(\alpha^{\phi}\otimes\id, w_{\alpha}^{\phi} \otimes 1)$
and $(\beta\otimes\id, w_{\beta}\otimes 1)$ are strongly Morita equivalent.
As $(\alpha^{\phi}, w_{\alpha}^{\phi})$ and $(\beta, w_{\beta})$ are strongly
Morita equivalent to $(\alpha^{\phi}\otimes\id, w_{\alpha}^{\phi}\otimes 1)$ and
$(\beta\otimes\id, w_{\beta}\otimes 1)$, respectively, we get that $(\alpha^{\phi}, w_{\alpha}^{\phi})$ and
$(\beta, w_{\beta})$ are strongly Morita equivalent, as asserted.
\end{proof}
%\par
%\bf
%Acknowledgement.
%\rm
%The author wishes to thank the referee for many valuable suggestions for improvement of
%the manuscript.


\begin{thebibliography}{99}

\bibitem{BC:discrete}E. B\'edos and R. Conti,
{\it On discrete twisted $C^*$-dynamical systems, Hilbert $C^*$-modules and regularity},
M{\"u}nster J. Math.
{\bf 5}
(2012), 183-208.

\bibitem{Blackadar:operator} B. Blackadar,
{\it Operator algebras. Theory of $C^*$-algebras and von Neumann algebras},
Encyclopaedia of Mathematical Sciences,
{\bf 122},
Operator Algebras and Non-commutative Geometry III, Springer-Verlag, Berlin, 2006.


\bibitem{BGR:linking}L. G. Brown, P. Green and M. A. Rieffel,
{\it Stable isomorphism and strong Morita equivalence of $C^*$-algebras},
Pacific J. Math.
{\bf 71}
(1977),
349--363.

\bibitem{BMS:quasi}L. G. Brown, J. Mingo and N-T. Shen,
{\it Quasi-multipliers and embeddings of Hilbert $C^*$-bimodules},
Can. J. Math.
{\bf 46}
(1994),
1150--1174.

\bibitem{Bui:crossed} H. H. Bui,
{\it Morita equivalence of twisted crossed products},
Proc. Amer. Math. Soc.
{\bf 123}
(1995),
2771-2776.


\bibitem{Choda:free} H. Choda,
{\it On freely acting automorphisms of operator algebras},
K\={o}dai Math. Sem. Rep.
{\bf 26}
(1974/75), 1-21.

\bibitem{Combes:morita}F. Combes,
{\it Crossed products and Morita equivalence},
Proc. London Math. Soc.
{\bf 49}
(1984),
289--306.

\bibitem{CMW:equivalence}R. E. Curto, P. S. Muhly and D. P. Williams,
{\it Cross products of strongly Morita equivalent $C^*$-algebras},
Proc. Amer. Math. Soc.
{\bf 90}
(1984),
528--530.

\bibitem{GVF:Noncommutative} J. M. Gracia-Bond\'{i}a, J. V\'{a}rilly and H. Figueroa,
{\it Elements of Noncommutative Geometry},
Birkh\"{a}user Boston, 2001.


\bibitem{JT:KK}K. K. Jensen and K. Thomsen,
{\it Elements of KK-theory},
Birkh$\ddot a$user,
1991.

\bibitem{Izumi:simple}M. Izumi,
{\it Inclusions of simple $C^*$-algebras},
J. reine angew. Math.
{\bf 547}
(2002), 97--138.


\bibitem{Kodaka:Picard2}K. Kodaka,
{\it The Picard groups for unital inclusions of unital $C^*$-algebras},
Acta Sci. Math. (Szged)
{\bf 86}
(2020), 183-207.


\bibitem{KT3:equivalence}K. Kodaka and T. Teruya,
{\it The strong Morita equivalence for coactions of a finite dimensional $C^*$-Hopf algebra on
unital $C^*$-algebras},
Studia Math.
{\bf 228}
(2015),
259--294.


\bibitem{KT4:morita}K. Kodaka and T. Teruya,
{\it The strong Morita equivalence for inclusions of $C^*$-algebras and conditional expectations for
equivalence bimodules}, J. Aust. Math. Soc.
{\bf 105}
(2018), 103--144.


\bibitem{KT5:inclusion2}K. Kodaka and T. Teruya,
{\it Coactions of a finite dimensional $C^*$-Hopf algebra on unital
$C^*$-algebras, unital inclusions of unital $C^*$-algebras and the strong Morita equivalence},
Studia Math.
{\bf 256}
(2021),
147--167.


\bibitem{PR:twisted} J. A. Packer and I. Raeburn,
{\it Twisted crossed products of $C^*$-algebras}, Math. Proc. Camb. Phil. Soc.
{\bf 106}
(1989), 293-311.

\bibitem{Pedersen:auto} G. K. Pedersen,
{\it $C^* $-algebras and their automorphism groups,}
Academic Press, London, New York, San Francisco, 1979.

\bibitem{RW:continuous} I. Raeburn  and D. P. Williams,
{\it Morita equivalence and continuous-trace $C^*$-algebras},
Mathematical Surveys and Monographs, vol. 60, American Mathematical Society, Providence, 1998.

\bibitem{Stormer:positive}E. St\o rmer,
{\it Positive linear maps of operator algebras},
Springer-Verlag, Berlin Heidelberg, 2013.


\bibitem{Zarikian:expectation}V. Zarikian,
{\it Unique expectations for discrete crossed products}, Ann. Funct. Anal.
{\bf 10}
(2019), no.1, 60-71.

\end{thebibliography}
\end{document}